\documentclass[12pt,oneside]{amsart}
\usepackage[utf8]{inputenc}
\usepackage[margin=1in]{geometry}
\usepackage{amscd}
\usepackage{hyperref}
\usepackage{fontenc}
\usepackage{textcomp}
\usepackage{comment}
\theoremstyle{plain}
\usepackage{enumitem}

\usepackage{amssymb,amsmath,amsthm}
\usepackage{mathrsfs}
\usepackage{xcolor}
\usepackage[all]{xy}
\usepackage{cancel}
\usepackage{graphicx}
\usepackage{lscape}
\usepackage{cite}

\newtheorem*{fact}{Fact}
\newtheorem*{facts}{Facts}
\newtheorem{thm}{Theorem}

\newtheorem{prop}{Proposition}

\newtheorem{lem}{Lemma}
\newtheorem{rema}{Remark}
\newtheorem{remasub}{Remark}

\newtheorem{defi}{Definition}

\newtheorem{coro}{Corollary}
\newtheorem*{nota}{Notation}

\newcommand\nn{\mathbb{N}}
\newcommand\zz{\mathbb{Z}}

\newcommand\rr{\mathbb{R}}
\newcommand\cc{\mathbb{C}}
\newcommand\kk{\mathbb{K}}

\newcommand\ff{\mathbb{F}}
\newcommand\ii{\mathbb{I}}
\newcommand\ungra{\mathbf{1}}

\newcommand\fra[2]{\displaystyle\frac{#1}{#2}}

\newcommand\tq{\mbox{ } | \mbox{ }}

\newcommand\opnorm{\Vert}

\newcommand\cali[1]{\mathcal{#1}}
\newcommand*\diff{\mathop{}\!\mathrm{d}}

\DeclareMathOperator{\sln}{SL}

\DeclareMathOperator{\cat}{CAT(-1)}
\DeclareMathOperator{\isom}{Isom}

\DeclareMathOperator{\supp}{supp}
\DeclareMathOperator{\rk}{rk}

\author[1]{Adrien Boyer}\thanks{Université Paris 7, aadrien.boyer@gmail.com}
\author[2]{Antoine Pinochet Lobos}\thanks{I2M, CNRS UMR7373, Université d'Aix-Marseille, antoine.pinochet-lobos@univ-amu.fr}
\author[3]{Christophe Pittet}\thanks{I2M, CNRS UMR7373, Université d'Aix-Marseille et Section de Mathématiques, Faculté des Sciences, Université de Genève, pittet@math.cnrs.fr\\
The authors acknowledge support of the FNS grant 200020-178828.}

\title{Radial rapid decay does not imply rapid decay}

\date{\today}

\begin{document}

\maketitle

\begin{abstract}
We provide a new, dynamical criterion for the radial rapid decay property. We work out in detail the special case of the group $\Gamma := \sln_2(A)$, where $A := \ff_q[X,X^{-1}]$ is the ring of Laurent polynomials with coefficients in $\ff_q$, endowed with the length function coming from a natural action of $\Gamma$ on a product of two trees, to show that is has the radial rapid decay (RRD) property and doesn't have the rapid decay (RD) property. The criterion also applies to irreducible lattices in semisimple Lie groups with finite center endowed with a length function defined with the help of a Finsler metric. These examples answer a question asked by Chatterji and moreover show that, unlike the RD property, the RRD property isn't inherited by open subgroups.
\end{abstract}

\tableofcontents

\section{Introduction}

The rapid decay property (RD), which can be stated as an inequality between two different norms on the convolution algebra of a group, was first introduced in \cite{HAAGERUP} and further developped in \cite{JOLISS}. It became a subject of great importance since V. Lafforgue discovered its connection with the Baum-Connes conjecture \cite{LAFFORGUE}. Any connected, semisimple Lie group has RD \cite{CHATPITSAL}, but it is an open question (asked in \cite{VALETTEBAUMCONNES} and now known as Valette's conjecture) to know whether cocompact lattices inherit the rapid decay property.

The radial rapid decay property (RRD) is a weakening of RD, first studied in \cite{VALETTE97}, and consists in restricting the RD inequality to the class of radial functions. The strategy of proof used in \cite{CHATPITSAL} for Lie groups, a reduction to radial functions, raised hope for a solution of Valette's conjecture when Perrone \cite{Perrone} managed to show that cocompact lattices have RRD. In this context, Chatterji asked for a group having RRD but not having RD \cite[p. 57]{chatterjird}. 

In this paper, we provide a sufficient, dynamical condition, for a group to have property RRD. We then study the case of the discrete group $\Gamma := \sln_2(A)$, where $A := \ff_q[X,X^{-1}]$ is the ring of Laurent polynomials with coefficients in $\ff_q$, that acts naturally on a product of trees, and on the product of the boundaries of these trees. Using the dynamical criterion, we prove that $\Gamma$ has RRD. Moreover, noticing that $\Gamma$ contains a lamplighter group as a subgroup, we prove that $\Gamma$ doesn't have RD, and give therefore a negative answer to Chatterji's question; finally, this example shows that RRD isn't inherited by open subgroups, whereas RD is. At first sight, this example may look surprising, since containing an amenable subgroup with exponential growth is a well-known obstruction to having RD.

\begin{nota} Throughout the paper, we will use the $\ll$ notation, as an alternative to the big $O$ notation. Precisely, \[f(n) \ll g(n)\]means that there is $M \in \rr$ such that for every sufficiently large $n$, we have that $\vert f(n) \vert \leq M \vert g(n) \vert$.
\end{nota}

\section{Statement of the results}

\subsection{Statement of the criterion}

\label{defi}

In this section, we give the necessary definitions and notation in order to state the dynamical criterion; we then investigate the range of application of the criterion.

\subsubsection{General framework}

Let $G$ be a locally compact group. Let $e\in G$ denote the identity element.

\begin{defi}[Length functions] A \textbf{length function} on $G$ is a map $L : G \rightarrow \rr_+$ such that
\begin{enumerate}
\item $L(e) = 0$;
\item $\forall g \in G$, $L(g^{-1}) = L(g)$;
\item $\forall g,h \in G$, $L(gh) \leq L(g) + L(h)$.
\end{enumerate}

If $E\subset G$ is any subset, we define, for all $t \in \rr_+$, $E_t=E\cap L^{-1}([0,t])$. A length function $L$ is said to be \textbf{proper} if $G_t$ is compact, for all $t$.
\end{defi}

\begin{nota} When it is unambiguous, we use the notation \[C_n := \{g \in G \tq L(g) \in [n,n+1)\}.\]If necessary, we add the reference to the group or the length function by writing $C^G_n$ or $C^L_n$.
\end{nota}

\begin{defi}[Radial functions] Let $L : G \rightarrow \rr_+$ be a proper length function. Denote by $C_c(G)$ the space of compactly-supported functions on $G$ (if $G$ is a discrete group, $C_c(G)$ is the space of finitely-supported functions and we denote it by $\cc[G]$).

We say that $f \in C_c(G)$ is \textbf{radial} if \[\forall g_1,g_2 \in G, \quad L(g_1) = L(g_2) \Rightarrow f(g_1) = f(g_2).\] Denote by $C^{rad}_c(G)$ the set of radial functions (if $G$ is a discrete group, $C^{rad}_c(G)$ is the space of radial functions of finite support, and we denote it by $\cc[G]^{rad}$).
\end{defi}

Let $\mu$ be a left Haar measure on $G$. Let $\rr[X]$ denote the algebra of polynomial functions in one variable $X$ and with real coefficients.

For $f \in C_c(G)$, let \[L(f) := \max\{ L(g) \tq g \in \supp(f)\},\] and for $f \in C_c(G)$ and $\xi \in L^2(G,\mu)$, consider the convolution \[f*\xi := \left(g \mapsto \int_{G} f(h) \xi(h^{-1}g) d\mu(h)\right).\]

Let us denote by $\Vert \cdot \Vert_{p \to q}$ the norm of a continuous operator between a $L^p$ and a $L^q$ space. Recall $\xi \mapsto f*\xi$ is a continuous linear operator on $L^2(G,\mu)$. For reasons of concision, let us denote \[\opnorm f \opnorm_{op} := \Vert \xi \mapsto f * \xi \Vert_{2\to 2}.\] We can now define the rapid decay property.

\begin{defi}[Rapid decay]
We say that $G$ has \textbf{property RD} with respect to $L$ if \[\exists P \in \rr[X],\ \  \forall f \in C_c(G),\ \  \opnorm f \opnorm_{op} \leq P(L(f)) \Vert f \Vert_2\] and we say that it has \textbf{radial property RD} with respect to $L$ if \[\exists P \in \rr[X],\ \  \forall f \in C^{rad}_c(G),\ \  \opnorm f \opnorm_{op} \leq P(L(f)) \Vert f \Vert_2.\]
\end{defi}

For further information on property RD, see \cite{chatterjird} and \cite{GARNCRD}.

The dynamical criterion asserts, for short, that if there is a suitable action of the group on some probability space, then the group has RRD. In order to state it precisely, we will need additional definitions and notation. Let $(B,\cali{T})$ be a measurable space, $G \curvearrowright B$ be a measurable action.

\begin{defi}[Quasi-invariant measure]
We say that a measure $\nu$ on $(B,\cali{T})$ is \textbf{quasi-invariant} if the action preserves $\nu$-null-sets, that is, for every $g \in G$, for every $C \in \cali{T}$ such that $\nu(C) = 0$, then $\nu(g C) = 0$.
\end{defi}

Let $\nu$ be a $\sigma$-finite quasi-invariant measure for the action $G \curvearrowright (B,\cali{T})$. We consider the following objects, which existence relies on the Radon-Nikodym theorem.

\begin{defi}[Radon-Nikodym cocycles, Koopman representation and Harish-Chandra function] 

We call \[\begin{array}{rcl}
c: G \times B &\rightarrow &\rr^*_+\\
(g,b) &\mapsto &\fra{\diff {g^{-1}}_* \nu}{\diff \nu}(b)\end{array}\]the \textbf{Radon-Nikodym cocycle}.

The formula \[\begin{array}{rcl}\pi : G&\rightarrow &\mathcal{U}(L^2(B ,\nu))\\
g&\mapsto &\left(h \mapsto (b \mapsto c(g^{-1},b)^\frac{1}{2}h(g^{-1}b)\right)\\
\end{array}\] defines a unitary representation called the \textbf{Koopman representation} associated to the action $\Lambda \curvearrowright (B,\nu)$. The associated \textbf{Harish-Chandra function} is defined as \[\Xi := g \mapsto \int_{B} c(g^{-1},b)^{\frac{1}{2}}\diff\nu(b) = \langle \pi(g)\ungra_B,\ungra_B\rangle.\]
\end{defi}

The main result of this paper is the following theorem.

\begin{thm}[Dynamical criterion for RRD]\label{shalomtrick}

Let $\Lambda$ be a discrete group with a proper length function $L$. Let $(B,\nu)$ be a $\sigma$-finite probability space. Let $\pi : \Lambda \rightarrow \mathcal{U}(L^2(B))$ be the Koopman representation arising from a measurable action $\Lambda \curvearrowright (B,\mu)$ leaving $\nu$ quasi-invariant, and let $\Xi$ be the corresponding Harish-Chandra function.

Assume that there is $M \in \rr$, $P \in \rr[X]$, such that
\begin{enumerate}
\item $\forall n \in \nn, \displaystyle\sup_{\gamma \in C_n} \Xi(\gamma) \leq \fra{P(n)}{\sqrt{\vert C_n \vert}}.$ 
\item $\forall n \in \nn, \left\Vert \fra{1}{\vert C_n\vert} \sum_{\gamma \in C_n} \fra{\pi(\gamma)}{\Xi(\gamma)} \right\Vert_{2\to 2} \leq M.$ 
\end{enumerate}

Then $\Lambda$ has RRD with respect to $L$.
\end{thm}

We will refer to the two hypotheses of the criterion as the \textbf{Harish-Chandra volume estimates condition} and the \textbf{uniform boundedness condition}. This dynamical criterion can be applied in a certain variety of cases, including irreducible lattices in semisimple Lie groups with finite center endowed with a length function defined with the help of a Finsler metric; see the comments section for a discussion about the Harish-Chandra estimates.

\subsubsection{Comments on the Harish-Chandra volume estimates condition}

\label{subsub: comments Harish-Chandra}

The Harish-Chandra estimates have been studied in the general setting of semisimple Lie groups, for which one can find an estimation that is, in some sense, sharp. Precisely, consider a connected semisimple Lie group $G$, with no compact factors and a finite center. Let $\mathfrak{g}$ be the Lie algebra of $G$. Let $\mathfrak{a}$ be a maximal abelian subalgebra of $\mathfrak{g}$. Let $\Sigma$ be the root system of $(\mathfrak{g},\mathfrak{a})$, let $\Sigma^+$ be a system of positive roots, let $\mathfrak{a}^+$ be the corresponding Weyl chamber, let $\Sigma^+_0$ be the set of indivisible positive roots. Let $G := KAN$ be the corresponding Iwasawa decomposition, let $M := K \cap Z(A)$, and let $P := MAN$ be the corresponding parabolic subgroup. Let $\rho$ be the half-sum of positive roots. Let us denote $A^+ := exp(\mathfrak{a}^+)$.

Finally, consider the action $G \curvearrowright G/P$, a quasi-invariant measure $\nu$ on $G/P$ and $\Xi$ the associated Harish-Chandra function. We have the following form of the Harish-Chandra estimate, due to Anker (see \cite{ANKER}):

for every $g \in G$, and any $KA^+K$-decomposition $g = k_1 \exp(H) k_2$, we have \[\Xi(g) \asymp \prod_{\alpha \in \Sigma^+_0} \left(1 + \alpha(H)\right)e^{-\rho(H)}.\]If $H \in \mathfrak{a}^+$, let us denote \[\Vert H \Vert := 2\rho(H).\]Let us endow $G$ with the following so-called Finsler length function: if $g = k_1\exp(H)k_2$ is a $KA^+K$ decomposition of any $g$ in $G$, let us denote $L(g) := \Vert H \Vert$. We then have the following estimate: \[\forall g \in G,\quad \Xi(g) \ll \left(1+L(g)\right)^{\vert \Sigma^+_0\vert}e^{-\frac{L(g)}{2}}.\]

Moreover, if we assume that $\Gamma$ is irreducible, we have the following estimate (by combining Prop 7.2 and 7.3 in \cite[p. 25]{ALBU}):

\[\vert C^\Gamma_n \vert \ll n^{\rk G - 1} e^n.\]

Therefore, we have \[\sup_{\gamma \in C^\Gamma_n} \Xi(\gamma) \sqrt{\vert C^\Gamma_n \vert} \ll (1+n)^{\Vert \Sigma^+_0 \vert \frac{(\rk G - 1)}{2}},\]

so that condition $(1)$ is satisfied.

\subsubsection{Comments on the uniform boundedness condition}

\label{subsub: comments uniform}

The following theorem, which generalizes ideas from \cite{BLP}, gives sufficient conditions for the uniform boundedness condition to hold.

\begin{thm}[Sufficient conditions for the uniform boundedness condition]\label{theorem from continuous to discrete} 

Let $G$ be a locally compact group endowed with a left Haar measure $\mu_G$, and a proper length function $\cali{L}$. Let $B$ be a compact space, $\nu$ be a Borel probability measure on $B$ and consider a measurable action of $G$ on $B$ leaving $\nu$ quasi-invariant. Let us assume that
\begin{itemize}[label=\textbullet]
\item the Radon-Nikodym cocycle and the Harish-Chandra function are continuous ;
\item the compact subgroup $G_0$ of $G$ consisting of elements of length $0$ is open and acts transitively on $B$ and leaves $\nu$ invariant.
\end{itemize}

Let $\Lambda$ be a lattice in $G$, satisfying the following growth assumption: if $L$ denotes the restriction of $\cali{L}$ to $\Lambda$, then there is $c \in \rr$ such that for all sufficiently large $n$, \[\mu_G(C^G_n) \leq c \vert C^\Lambda_n \vert.\]

Let us finally denote by $\pi$ the Koopman representation of $\Lambda$ associated with the (restricted) action $\Lambda \curvearrowright B$. Then there is $M \in \rr$ such that \[\forall n \in \nn, \left\Vert \fra{1}{\vert C^\Lambda_n\vert} \sum_{\gamma \in C^\Lambda_n} \fra{\pi(\gamma)}{\Xi(\gamma)} \right\Vert_{2\to 2} \leq M.\]

\end{thm}

The uniform boundedness condition has been studied in \cite{BAMU}, \cite{BOYER}, \cite{GARNBOUND}, \cite{BLP} and \cite{BOYPIN} where authors investigate generalizations of the von Neumann ergodic theorem to the situation where the measure is only quasi-invariant. In several cases of interest, the relevant generalization of von Neumann means are the normalized means \[\fra{1}{\vert C_n \vert} \sum_{\gamma \in C_n} \frac{\pi(\gamma)}{\Xi(\gamma)}\]which are shown to converge, in the weak-operator topology, to the orthogonal projector on the constant functions subspace, using the fact that the sequence of these means is uniformly operator-norm-bounded. In particular, the following are true:

\begin{facts} The uniform boundedness condition is true for $(\Lambda,L,B,\nu)$ in the following situations:
\begin{itemize}
\item \cite{BAMU} if $\Lambda$ is the fundamental group of a compact, negatively curved manifold $X$ with universal cover $\tilde{X}$, such that $L$ is the length function associated to the action of $\Gamma$ on $\tilde{X}$ and a fixed base-point $x_0 \in \tilde{X}$, and $B$ is the Gromov boundary of $\tilde{X}$ endowed with the Patterson-Sullivan measure $\nu$ associated to $x_0$;
\item \cite{BOYPIN} if $\Lambda$ is a free group over a finite set of generators, where $L$ is the length function associated to the action of $\Lambda$ on its Cayley tree (with respect to a generating basis) $\tilde{X}$ and $B$ is the Gromov boundary of $\tilde{X}$ and $\nu$ is the Patterson-Sullivan measure;
\item \cite{BOYER} if $\Lambda$ is a convex cocompact discrete group of isometries of a $\cat$-space $X$ with non-arithmetic spectrum and a finite BMS measure, and $L$ is the length function associated to the action and $B$ is the Gromov boundary and $\nu$ is the Patterson-Sullivan probability measure;
\item \cite{BLP} if $G$ is a noncompact, connected, semisimple Lie group with finite center, let us keep the notation from section \ref{subsub: comments Harish-Chandra}: $\Lambda$ is a lattice in $G$, $B := G/P$, $\nu$ is a quasi-invariant probability measure and $L$ is defined by the formula \[\forall k_1,k_2 \in K,\ \forall H \in \mathfrak{a}^+,\quad L(k_1exp(H)k_2) := \Vert H \Vert\]where $\Vert \cdot \Vert$ is a norm induced from a scalar product on $\mathfrak{a}$. Moreover, we know from \cite{GORNEV} (see \cite[Théorème 1.4.41]{APL} for details) that if $\Lambda$ is irreducible, the growth assumption of Theorem \ref{theorem from continuous to discrete} is true for every length integer-valued proper length function $L$ on $G$, up to a rescaling, i.e. replacing $L$ by $L' := \lceil \frac{L}{\alpha} \rceil$ for $\alpha$ large enough so that the elements in $G$ of length $\leq \alpha$ generate $G$. Therefore, the uniform boundedness condition is true for $(\Lambda,L,B,\nu)$ if $L$ is a suitable rescaling of the length function defined in section \ref{subsub: comments Harish-Chandra} with the help of a Finsler metric.
\end{itemize}
\end{facts}

\subsection{A group having RRD but not having RD}

In this section, we apply the previous results to a specific example and exhibit a group having RRD but not having RD for a natural length function.

\subsubsection{The group $\sln_2(A)$ and its action on the product of two trees}

\label{defi Gamma G}

Let $q$ be a power of a prime number and $A := \ff_q[X,X^{-1}]$ be the commutative ring of Laurent polynomials in the variable $X$ with coefficients in $\ff_q$, the field with $q$ elements. It is the subring of $\kk := \ff_q(X)$ (the ring of rational fractions over the field $\ff_q$) generated by the elements $X$ and $X^{-1}$ and its additive group is the $\ff_q$-vector space with basis $\{X^n \tq n \in \zz\}$, so each element of $A$ can be written as $\sum_{n \in \zz} a_n X^n$, such that $\forall n \in \zz$, $a_n \in \ff_q$ and all but a finite number of the $a_n$ being zero.

Let us now define on $\kk$ two valuations: if $F \in \kk$, there are $n \in \zz$, $P,Q \in \ff_q[X]$ such that $F := X^n P/Q$ with $X \nmid P$ and $X \nmid Q$. The integer $n$ only depends on $F$ and is denoted by $v_{0} (F)$ (it is the valuation at the place $0$). If $F = P/Q \in \kk$, we define $v_{\infty} (F)$ to be the number $\deg P - \deg Q$ (it is the valuation at infinity). We denote, for $i \in \{0,\infty\}$, $\vert \cdot \vert_i := q^{-v_i(\cdot )}$ the associated norms, and we consider the corresponding completions $\mathbb{K}_0$ and $\mathbb{K}_\infty$. The elements of the ring $A$ are called the $\{0,\infty\}$-integral elements of $\mathbb{K}$. Now consider the diagonal embedding $\Delta : \sln_2(A) \hookrightarrow \sln_2(\kk_0) \times \sln_2(\kk_\infty)$. According to \cite[p. 1]{MARGU}, \[\Gamma := \Delta(\sln_2(A)) \subset G := \sln_2(\kk_0) \times \sln_2(\kk_\infty)\] is a lattice in $G$.

Now, both $\sln_2(\mathbb{K}_0)$ and $\sln_2(\mathbb{K}_\infty)$ act simplicially and properly on their Bruhat-Tits trees $T_0$ and $T_\infty$ (see \cite[p. 69]{SERRE} for the detailed construction of the tree and the action on it) which both happen to be $(q+1)$-regular trees, so the group $G = \sln_2(\mathbb{K}_0) \times \sln_2(\mathbb{K}_\infty)$ (and therefore, $\Gamma$ !) acts cellularly on the product $\mathbb{I} := T_0 \times T_\infty$. Let us denote $V(T_0)$ and $V(T_\infty)$ the sets of vertices of these two trees. Let us denote by $d_0$ and $d_\infty$ the distances giving length $1$ to the edges of $T_0$ and $T_\infty$. We define on $\mathbb{I} := T_0 \times T_\infty$, the so-called $L^1$ distance, that is, $$\forall x,y \in T_0,\ \ \forall x',y' \in T_\infty,\ \ d_\ii\left((x,y),(x',y')\right) := d_0(x,y) + d_\infty(x',y').$$

For this distance function, the group $G$ (and, in particular, $\Gamma$) acts on $\mathbb{I}$ (and also on $V(T_0) \times V(T_\infty)$) by isometries. 

Let $(v_0,v_\infty) \in V(T_0) \times V(T_\infty)$. We define two length functions as follows: for $i \in \{0,\infty\}$, let us denote \[\forall g_i \in G_i,\quad L_i(g_i) := d_i(v_i,g_iv_i);\] finally, we define a length function on $G$ as follows: $\forall g := (g_0,g_\infty) \in G$, \[L(g) := d_\ii\left((v_0,v_\infty),\gamma (v_0,v_\infty)\right) = L_0(g_0) + L_\infty(g_\infty).\] 

According to \cite{LUBMOZRAG}, the length function $L$ on $\Gamma$ is quasi-isometric to any of the word-lengths on $\Gamma$.

We perform the necessary computations and apply the criterion to $\Gamma$ and the action of $\Gamma$ on $B$, the product of the boundaries of the Bruhat-Tits trees to deduce the following corollary.

\begin{coro}\label{coro un} The group $\sln_2(A)$ has RRD with respect to $L$.
\end{coro}

\subsubsection{Two general consequences}

As mentioned in the introduction, one way to prove that a discrete group does not have RD is to prove that it contains an amenable subgroup of exponential growth, since an open subgroup of a group having RD has it as well, and since a finitely-generated amenable group has RD with respect to any of its word-lengths if and only if it is of polynomial growth. However, the example of $\sln_2(A)$ shows that the analogue obstruction for RRD does not hold, by noticing that $\sln_2(A)$ contains a lamplighter group as a subgroup. These observations are the subject of the following corollary.

\begin{coro}\label{coro trois}
There exists finitely generated groups endowed with lengths functions which are quasi-isometric to word-lengths
\begin{enumerate}
\item which have RRD but do not have RD;
\item which contain subgroups that do not have RRD for some word-lengths.
\end{enumerate}
\end{coro}

\begin{rema} No non-uniform lattice in a semisimple Lie group of rank at least two has RD, since they have $U$-elements \cite{LUBMOZRAG}. As explained in Sections \ref{subsub: comments Harish-Chandra} and \ref{subsub: comments uniform}, any irreducible lattice of a semisimple Lie group $G$ of finite center, endowed with a rescaled Finsler length function has RRD. Therefore, any irreducible, non-uniform lattice in a semisimple Lie group with finite center of rank at least two has property $(1)$ and $(2)$ of Corollary \ref{coro trois}.

In this paper, we prove $(1)$ and $(2)$ in the case of $\sln_2(\ff_q[X,X^{-1}])$, by means of elementary computations.
\end{rema}

\subsection{Structure of the paper}
We prove the dynamical criterion (Theorem \ref{shalomtrick}) in Section \ref{proof criterion}. We prove Theorem \ref{theorem from continuous to discrete} in Section \ref{section proof theorem continuous discrete}, and we prove in Section \ref{radialrd} the fact that $\Gamma$ has RRD (Corollary \ref{coro un}) in which we prove that we can apply the dynamical criterion to $\Gamma$. In Section \ref{nord}, we recall some facts on the RD property, exhibit the lamplighter subgroup $H$ of $\Gamma$ and prove that $\Gamma$ does not have RD and that $H$ does not have RRD (Corollary \ref{coro trois}).

\section{Proof of the criterion (Theorem \ref{shalomtrick})}\label{proof criterion}

Let $\Lambda$ be a discrete group, and $L$ be a proper integer-valued length function on $\Lambda$. Let us define \[\forall \gamma \in \Lambda,\ \ \ungra_{n}(\gamma) := \left\{\begin{array}{rl}
1 &\mbox{if } L(\gamma) = n\\
0 &\mbox{else}\\
\end{array}\right.\]The following proposition shows that for integer-valued length functions, rapid decay on spheres implies radial rapid decay.

\begin{prop}\label{integerlength} Let $\Gamma$ be a discrete group, and $L$ an integer-valued proper length function on $\Gamma$.  Then if $$\exists P \in \rr[X],\ \ \forall n \in \nn,\ \ \opnorm \ungra_n \opnorm_{op} \leq P(n) \Vert \ungra_n \Vert_2,$$ then $\Gamma$ has RRD with respect to $L$.
\end{prop}

\begin{proof}

Let $f \in \cc[\Gamma]^{rad}$. We have $f = \sum_{n \in \nn} a_n \ungra_{n}$ where $a_n$ is the common value of $f$ on elements of length $n$. Notice that for all but finitely many $n$, $a_n = 0$. Assuming there is a $P \in \rr[X]$ as in the hypotheses, choose $Q \in \rr[X]$ positive and non-decreasing on $\rr_+$ such that $(1+t)^2 (P(t))^2 \leq Q(t)$ for all $t \in \rr_+$. Now, we have
$$\begin{array}{rcl}
\opnorm f \opnorm_{op} &\leq &\sum_{n \in \nn} \opnorm a_n \ungra_{n} \opnorm_{op}\\
\\
&\leq &\sum_{n \in \nn} P(n) \Vert a_n \ungra_{n} \Vert_2\\
{\small \mbox{(Cauchy-Schwarz inequality)}} &\leq &\left(\sum_{n \in \nn} (1+n)^2 (P(n))^2 \Vert a_n \ungra_{n} \Vert^2_2\right)^\frac{1}{2}\left(\sum_{n \in \nn} \left(1+n\right)^{-2}\right)^{\frac{1}{2}}\\
&\leq &C\cdot \left(\sum_{n \in \nn} Q(n)\Vert a_n \ungra_{n} \Vert^2_2\right)^\frac{1}{2}\\
&\leq &C\cdot \sup\left(\left\{Q(n) \tq a_n \not = 0\right\}\right)\left(\sum_{n \in \nn} \Vert a_n \ungra_{n} \Vert^2_2\right)^\frac{1}{2}\\
&= &C\cdot Q(L(f))\left(\sum_{n \in \nn} \Vert a_n \ungra_{n} \Vert^2_2\right)^\frac{1}{2}\\
&= &C\cdot Q(L(f)) \Vert f \Vert_2\\
\end{array}$$ so $\Gamma$ has radial property RD with respect to $L$.
\end{proof}

We are now ready to prove Theorem \ref{shalomtrick}. Let us recall that we are given $(B,\nu)$, a $\sigma$-finite probability space, on which $\Lambda$ acts measurably, and we assume that $\nu$ is quasi-invariant. We denote by $\pi : \Gamma \rightarrow \mathcal{U}(L^2(B))$ the associated Koopman representation and by $\Xi$ be the corresponding Harish-Chandra function and that we assume that there is $M \in \rr$, $P \in \rr[X]$, such that
\begin{enumerate}
\item $\forall n \in \nn, \displaystyle\sup_{\gamma \in C_n} \Xi(\gamma) \leq \fra{P(n)}{\sqrt{\vert C_n \vert}}.$ 
\item $\forall n \in \nn, \left\Vert \fra{1}{\vert C_n\vert} \sum_{\gamma \in C_n} \fra{\pi(\gamma)}{\Xi(\gamma)} \right\Vert_{2\to 2} \leq M.$
\end{enumerate}

\begin{proof}[Proof of Theorem \ref{shalomtrick}] Let us first notice that $\Vert \ungra_{C_n} \Vert_2 = \sqrt{\vert C_n \vert}$.
According to Proposition \ref{integerlength}, it is enough to prove \[\exists  P \in \rr[X], \forall n \in \nn, \quad \Vert \ungra_{C_n} \Vert_{op} \leq P(n) \sqrt{\vert C_n \vert}.\]
We will in fact prove \[\exists  P \in \rr[X], \forall n \in \nn, \quad \Vert \pi(\ungra_{C_n}) \Vert_{2\to 2} \leq P(n) \sqrt{\vert C_n \vert}.\] This is enough, because according to \cite[Lemma 2.3]{Shalom2000} (we can apply this lemma, since its hypotheses are satisfied, because $\ungra_B$ is an obvious positive vector), we have that \[\forall n \in \nn, \opnorm \textbf{1}_{C_n} \opnorm_{op} \leq \Vert \pi(\textbf{1}_{C_n}) \Vert_{2\to 2}.\]

We claim that \[\forall n \in \nn, \left\Vert \fra{1}{\vert C_n \vert} \sum_{\gamma \in C_n} \pi(\gamma) \right\Vert_{2\to 2} \leq \sup_{\gamma \in C_n} \Xi(\gamma) \left\Vert \fra{1}{\vert C_n \vert} \sum_{\gamma \in C_n} \fra{\pi(\gamma)}{\Xi(\gamma)} \right\Vert_{2\to 2}.\]
Notice that the claim, combined with (1) and (2), ends the proof of the Theorem.

Let us now prove the claim. If $h \in L^2(B)$, denote by $h_r$ and $h_i$ its real and imaginary parts and let, $\forall a \in \{r,i\}$, $h^+_a := \max(h_a,0)$ and $h^-_a := h^+_a - h_a$. Then $h^\pm_r$ and $h^\pm_i$ are all positive $L^2$ functions, and $\max\{\Vert h^+_r \Vert_2,\Vert h^-_r \Vert_2,\Vert h^+_i \Vert_2,\Vert h^-_i \Vert_2\} \leq \Vert h \Vert_2$. Let us denote, until the end of the proof, \[M_n := \fra{1}{\vert C_n \vert} \sum_{\gamma \in C_n} \pi(\gamma)\] and \[M^{\Xi}_n := \displaystyle\sup_{\gamma \in C_n} \Xi(\gamma)\fra{1}{\vert C_n \vert} \sum_{\gamma \in C_n} \fra{\pi(\gamma)}{\Xi(\gamma)}.\]

We then have \[\begin{array}{rcl}
\Vert M_n(h) \Vert^2_2 &= &\displaystyle \Vert M_n(h^+_r) \Vert^2_2 + \Vert M_n(h^-_r) \Vert^2_2 + \Vert M_n(h^+_i) \Vert^2_2 + \Vert M_n(h^-_i) \Vert^2_2\\
&\leq &\displaystyle \Vert M^\Xi_n(h^+_r) \Vert^2_2 + \Vert M^\Xi_n(h^-_r) \Vert^2_2 + \Vert M^\Xi_n(h^+_i) \Vert^2_2 + \Vert M^\Xi_n(h^-_i) \Vert^2_2\\
&= &\Vert M^{\Xi}_n h \Vert^2_2\\
\end{array}\] where the inequality comes from the fact that for every nonnegative function $f$, we have \[0 \leq M_n(f) \leq M^\Xi_n(f).\]This proves our claim.
\end{proof}

\section{Proof of Theorem \ref{theorem from continuous to discrete}}

\label{section proof theorem continuous discrete}

In this section, we prove Theorem \ref{theorem from continuous to discrete}. Let us recall the context. Let $G$ be a locally compact group with left Haar measure $\mu_G$ and $L$ be a proper integer-valued length function on $G$, $B$ be a compact space, $\nu$ be a Borel probability measure on $B$, and $\Lambda$ be a lattice in $G$. Let us denote, for every $n \in \nn$, \[C^G_n := \{g \in G \tq L(g) = n\}\]and\[C^\Lambda_n := \{g \in G \tq L(g) = n\}.\]We consider a continuous action of $G$ on $B$ which leaves $\nu$ quasi-invariant.

We make the following assumptions.

\begin{itemize}[label = \textbullet]
\item there is a constant $c$ such that for every sufficiently large $n$, $\mu_G(C^G_n) \leq c \vert C^\Lambda_n \vert$ ;
\item the Radon-Nikodym cocycle and the Harish-Chandra function are continuous ;
\item the compact subgroup $G_0$ acts transitively on $B$.
\end{itemize}

Under these assumptions, we are going to prove that there is a constant $M \in \rr_+$ such that \[\forall n \in \nn, \left\Vert \fra{1}{\left\vert C^\Lambda_n\right\vert} \sum_{\gamma \in C^\Lambda_n} \fra{\pi(\gamma)}{\Xi(\gamma)} \right\Vert_{2 \to 2} \leq M.\]

Let us define the averages on $\Gamma$ and on $G$:
\[M^G_n := \fra{1}{\mu(C^G_n)}\int_{C^G_n} \fra{\pi(g)}{\Xi(g)} \diff \mu_G(g) \in \mathcal{B}(L^2(B))\] and
\[M^\Lambda_n := \fra{1}{\vert C^\Lambda_n \vert} \sum_{\gamma \in C^\Lambda_n} \fra{\pi(\gamma)}{\Xi(\gamma)} \in \mathcal{B}(L^2(B)).\]

The goal is to prove that the family $(M^\Lambda_n)_{n \in \nn}$ is bounded in $\cali{B}(L^2(B))$ and the strategy is to prove the following chain of inequalities:

\[\Vert M^\Lambda_n \Vert_{2 \to 2} \stackrel{(1)}{\leq} \Vert M^\Lambda_n \Vert_{\infty \to \infty} \stackrel{(2)}{=} \Vert M^\Lambda_n \textbf{1}_B \Vert_{\infty} \stackrel{(3)}{\ll} \Vert M^G_n \textbf{1}_B \Vert_\infty \stackrel{(4)}{=} 1.\]

Let us first state a version of the Riesz-Thorin theorem we need (a proof of the reduction of the lemma to the general Riesz-Thorin theorem can be found in \cite[Proposition 2.8]{BOYPIN}).

\begin{lem}\label{rieszthorinade} Let $(X,m)$ be a probability space. Let $T$ be a continuous operator $L^1(X,m) \rightarrow L^1(X,m)$ such that \begin{itemize}[label=\textbullet]
\item the restriction of $T$ to $L^2(X,m)$ induces a continuous self-adjoint operator on $L^2(X,m)$;
\item the restriction of $T$ to $L^\infty(X,m)$ induces a continuous operator on $L^\infty(X,m)$.
\end{itemize}

Then we have \[\Vert T \Vert_{2 \to 2} \leq \Vert T \Vert_{\infty \to \infty}.\]
\end{lem}

\begin{lem}\label{norme infinie realisee} We have that $\Vert M^{\Lambda}_n \Vert_{\infty \to \infty} = \Vert M^{\Lambda}_n \ungra_{B} \Vert_\infty$.
\end{lem}

\begin{proof} Let $h \in L^\infty(B)$. From the pointwise inequality (valid almost everywhere) \[- \opnorm h \opnorm_\infty \ungra_{B} \leq h \leq \opnorm h \opnorm_\infty \ungra_{B}\] we deduce the pointwise inequality (valid almost everywhere) \[- \Vert h \Vert_\infty M^{\Lambda}_n \ungra_{B} \leq M^{\Lambda}_n h \leq \Vert h \Vert_\infty M ^{\Lambda}_n \ungra_{B}.\] So, we get $\Vert M^\Lambda_n h \Vert_\infty \leq \Vert h \Vert_\infty \Vert M^\Lambda_n \ungra_{B} \Vert_\infty$, so $\Vert M^\Lambda_n \Vert_{\infty \to \infty} \leq \Vert M^\Lambda_n \ungra_{B} \Vert_\infty$.
\end{proof}

\begin{lem}\label{stability}
Let $G$ be a locally compact group, acting on a probability space $(B,\nu)$ such that the Radon-Nikodym cocycles are continuous. Let us recall that $\Lambda$ is a discrete subgroup of $G$. Then, there exist a relatively compact neighborhood $U$ of $e$ in $G$ and non-zero constants $C_1,C_2$ such that
\begin{enumerate}
\item $\Lambda \cap U = \{e\}$;
\item $\forall u \in U$, $\forall g \in G$, $\forall b \in B$, 
$C_1 (\pi(g) \ungra_{B}(b))^2=C_1c(g^{-1},b)\leq c((gu)^{-1},b)=(\pi(gu) \ungra_{B}(b))^2$;
\item $\forall u \in U$, $\forall g \in G$, $\Xi(gu)\geq C_2\Xi(g)$.
\end{enumerate}
\end{lem}

\begin{proof} Let $V$ be a symmetric compact neighborhood of $e$ in $G$. Since $c$ is continuous, it reaches its minimum $C_1$ and its maximum $C_2$ on the compact $V\times B$. We choose a symmetric neighborhood of the identity $W$ which trivially intersects $\Gamma$ and take $U := V \cap W$.  The cocycle identity 
\[
\forall b \in B,\ \  v \in V,\ \ 1 = c(v^{-1}v,b) = c(v^{-1},vb)c(v,b)
\] 
and the symmetry of $V$ imply that $C_1C_2 = 1$. 
Since \[\forall g\in G,\ \  \forall b\in B,\ \  \forall v \in V,\ \ C_1 c(g,b) \leq c(vg,b) \leq C_2 c(g,b),\] the left inequality gives us $(1)$ and the right inequality, used in the following calculation \[\begin{array}{rcl}\forall v\in V,\ \ \forall g \in G,\ \  \Xi(gv) &= &\Xi(v^{-1}g^{-1}) = \int_B c(v^{-1}g^{-1},b)^{\frac{1}{2}} \diff \mu(b) \\
&\leq &\int_B C_2c(g^{-1},b)^{\frac{1}{2}}\diff \mu(b) = C_2\Xi(g^{-1}) = C_2\Xi(g),
\end{array}\] allows us to prove $(2)$. 
\end{proof}

\begin{prop} \label{discret-au-continu}There exists a relatively compact neighborhood $U$ of $e$ in $G$ and a constant $C$ such that for every finite subset $A \subset \Lambda$, \[\sum_{\gamma \in A} \fra{\pi(\gamma) \ungra_{B}}{\Xi(\gamma)} \leq C \int_{AU} \fra{\pi(g)\ungra_{B}}{\Xi(g)} \diff\mu_G(g).\]
\end{prop}

\begin{proof}

Let $U$ be a compact neighborhood of $e$ in $G$ as in Lemma \ref{stability}.

Let $\gamma \in \Lambda$ and $b \in B$. Then \[\begin{array}{rcl}
\fra{\pi(\gamma)\ungra_{B}(b)}{\Xi(\gamma)}  &= &\fra{1}{\mu(U)}\int_{U} \fra{\pi(\gamma)\ungra_{B}(b)}{\Xi(\gamma)}du\\
&\leq &\fra{1}{\mu(U)}\int C_1 \fra{\pi(\gamma u)\ungra_{B}(b)}{\Xi(\gamma)}du\\
&\leq &\fra{C_1}{\mu(U)} \int_U \fra{\pi(\gamma u)\ungra_{B}(b)}{\frac{\Xi(\gamma u)}{C_2}} du\\
&\leq &\fra{C_1 C_2}{\mu(U)} \int_{\gamma U} \fra{\pi(u)\ungra_{B}(b)}{\Xi(u)}  du\\
\end{array}\]

Now, a summation over all $\gamma \in A$ ends the proof.
\end{proof}

\begin{prop}\label{coro discret au continu} There is a constant $C$ such that for all $n \in \nn$, and all $b \in B$, \[0 \leq M^\Lambda_n \ungra_{B}(b) \leq C M^G_{n} \ungra_{B}(b).\] In particular, we have that \[\Vert M^\Lambda_n \ungra_{B} \Vert_\infty \ll \Vert M^G_n \ungra_{B} \Vert_\infty.\]
\end{prop}

\begin{proof} Take a neighborhood $U$ of $e$ in $G$ and $C$ given by Proposition \ref{discret-au-continu}. We can assume that $\forall g \in U$, $L(g) = 0$. We have $C^\Lambda_n U \subseteq C^G_n$, so, by Proposition \ref{discret-au-continu}, we have, for all $n$, \[\sum_{\gamma \in C^\Lambda_n} \fra{\pi(\gamma) \ungra_{B}}{\Xi(\gamma)} \leq C \int_{C^G_n} \fra{\pi(g)\ungra_{B}}{\Xi(g)} \diff\mu_G(g)\] and by Lemma \ref{growthestimates}, there is $m > 0$ such that for all $n$, \[\fra{1}{\vert C^\Lambda_n \vert}\sum_{\gamma \in C^\Lambda_n} \fra{\pi(\gamma) \ungra_{B}}{\Xi(\gamma)} \leq \fra{C}{m\mu\left(C^G_n\right)}\int_{C^G_n} \fra{\pi(g)\ungra_{B}}{\Xi(g)}  \diff\mu_G(g)\] which gives the claim.
\end{proof}

\begin{lem}\label{constante vaut un} The function $M^G_n \ungra_B$ is constant and equal to $1$.
\end{lem}

\begin{proof} Let $k \in G_0$. First of all, let us notice that we have, for all $b \in B$, \[M^G_n \ungra_{B}(b) = M^G_n \ungra_{B}(kb).\]To prove it, let us make the following computation using the cocycle property and the fact that $G_0$ preserves $\mu$, so that $c(k^{-1},\cdot) \equiv 1$:
\[\begin{array}{rcl}
M^G_n(\ungra_B)(kb) &= &\displaystyle\int_{C^G_n} \left(\pi(g) \ungra_B\right)(kb) \diff \mu_G(g)\\
&= &\displaystyle\int_{C^G_n} c(g^{-1},kb)^{\frac{1}{2}} \ungra_B(kb) \diff \mu_G(g)\\
&\stackrel{h = k^{-1}g}{=} &\displaystyle\int_{C^G_n} c(h^{-1}k^{-1},kb)^{\frac{1}{2}} \ungra_B(h^{-1}b)\diff\mu_G(g)\\
&= &\displaystyle\int_{C^G_n} c(h^{-1},k^{-1}kb)^{\frac{1}{2}}c(k^{-1},b)^{\frac{1}{2}} \ungra_B(h^{-1}b)\diff \mu_G(g)\\
&= &\displaystyle\int_{C^G_n} c(h^{-1},b)^{\frac{1}{2}}\ungra_B(h^{-1}b)\diff \mu_G(g)\\
&= &\displaystyle M^G_n(\ungra_B)(b).\\
\end{array}\]Moreover, recall that, by assumption, the action of $G_0$ on $B$ is transitive, so the function $M^G_n\ungra_B$ is constant and equal to $c \in \rr$.

We integrate and use Fubini's theorem to prove that $c$ is in fact $1$:\[\begin{array}{rcl}
c &= &\displaystyle\int_B M^G_n \ungra_B\\
&= &\fra{1}{\mu(C^G_n)}\int_B \int_{C^G_n} \fra{c(g^{-1},b)^{\frac{1}{2}}}{\Xi(g)} \ungra_B(g^{-1}b) \diff \mu_G(g) \diff\mu(b)\\
&= &\fra{1}{\mu(C^G_n)}\int_{C^G_n} \int_B \fra{c(g^{-1},b)^{\frac{1}{2}}}{\Xi(g)} \diff\mu(b)\diff \mu_G(g)\\
&= &\fra{1}{\mu(C^G_n)}\int_{C^G_n}\frac{\Xi(g^{-1})}{\Xi(g)}\diff \mu_G(g)\\
&= &1\\
\end{array}\]
\end{proof}

\begin{proof}[Proof of Theorem \ref{theorem from continuous to discrete}]

Let us recall the chain of inequalities we want to prove:
\[\Vert M^\Lambda_n \Vert_{2 \to 2} \stackrel{(1)}{\leq} \Vert M^\Lambda_n \Vert_{\infty \to \infty} \stackrel{(2)}{=} \Vert M^\Lambda_n \textbf{1}_B \Vert_{\infty} \stackrel{(3)}{\ll} \Vert M^G_n \textbf{1}_B \Vert_\infty \stackrel{(4)}{=} 1.\]

Inequality $(1)$ is a simple application of the Riesz-Thorin theorem (Lemma \ref{rieszthorinade}). Equality $(2)$ is proved in Lemma \ref{norme infinie realisee}. Inequality $(3)$ is proved in Proposition \ref{coro discret au continu} and equality $(4)$ is proved in Lemma \ref{constante vaut un}.
\end{proof}

\section{Proof of Corollary \ref{coro un}}

\label{radialrd}

\subsection{Structure of the proof}

Let us recall that in Section \ref{defi Gamma G}, we defined a ring $A$, fields $\kk_0$ et $\kk_\infty$, two $(q+1)$-regular trees $T_0$ and $T_\infty$, a distance $d_{\mathbb{I}}$ on the product of these trees, an action of $G := \sln_2(\kk_0) \times \sln_2(\kk_\infty)$ on $T_0 \times T_\infty$, a length function $L$ on $G$, and $\Gamma$ to be the image, under the diagonal embedding, of $\sln_2(A)$ in $G$.

The proof of Corollary \ref{coro un} goes as follows.

\begin{proof}[Proof of Corollary \ref{coro un}]

Let $B$ be the compact space, $\nu$ be the Borel probability measure on $B$, defined in Subsection \ref{subsubsection: A useful unitary representation}, which also defines a measurable action $G \curvearrowright B$ leaving $\nu$-invariant. Let us denote by $\pi$ the Koopman representation of $G$ associated to this action. According to Proposition \ref{prop: quasi-invariant and continuous cocycle} and Proposition \ref{lemma: HC}, the Radon-Nikodym and the Harish-Chandra function are continuous, and according to Proposition \ref{groupgrowth}, the growth estimate assumption in Theorem \ref{theorem from continuous to discrete} is satisfied. We deduce then from Theorem \ref{theorem from continuous to discrete} that \[\left\Vert\fra{1}{\vert C^\Gamma_n \vert} \sum_{\gamma \in C^\Gamma_n} \fra{\pi(\gamma)}{\Xi(\gamma)} \right\Vert_{2 \to 2} \ll 1.\]Moreover, we deduce from Proposition \ref{first ingredient shalom trick} that there is a polynomial $P$ such that \[\sup_{\gamma \in C^\Gamma_n} \Xi(\gamma) \sqrt{\vert C^\Gamma_n \vert} \leq P(n).\]Therefore, the two assumptions of the dynamical criterion are satisfied; so $\Gamma$ has RRD with respect to $L$.
\end{proof}

In the following subsections, we prove all the propositions used in the proof.

\subsection{The Koopman representation on the product of the boundaries}\label{subsubsection: A useful unitary representation}

Here we recall the construction of the boundary of a tree in order to build a useful compact space $B$ on which $G$ acts and endow it with a quasi-invariant Borel probability measure.

Let $T$ be a $d$-regular tree. Let us recall a few facts from the theory of $\cat$ spaces and measures at infinity (see \cite{BOURD} for details).

Let us denote by $\partial T$ the set of equivalence classes of asymptotic rays in $T$ (the equivalence class of $r$ is denoted by $r(+\infty)$). Fixing a point $x \in T$, we can consider $\partial_x T$, the set of geodesic rays starting at $x$. The quotient map $\partial_x T \rightarrow \partial T$ can be shown to be a bijection. The image of the topology of uniform convergence on compact sets on $\partial_x T$ is a topology on $\partial T$ that, in fact, does not depend on $x$. For this topology, $\partial T$ is homeomorphic to a Cantor set. 

Now, the set $\overline{T} = T \cup \partial T$ can be endowed with a natural topology which makes $\overline{T}$ a compactification of $T$ and induces on $\partial T$ the topology defined above. This compactification has the important property that every isometry $\gamma$ of $T$ extends to a homeomorphism of $\overline{T}$, the restriction of which to $\partial T$ being denoted by $\partial \gamma$.

If $x,y,z \in T$, we denote $(y|z)_x := \frac{1}{2}\left(d(x,y) + d(x,z) - d(y,z)\right)$ and we call it the \textbf{Gromov product}. If $x \in T$, then $(.|.)_x$ can be extended in a continuous manner to $\overline{T}^2$, which we again call the Gromov product, and if we set, for all $b,b' \in \partial T$, $d_x(b,b') := e^{-(b|b')_x}$, then $d_x$ is a distance on $\partial T$ which also induces the topology defined above. However, the distances $d_x$ do depend on $x$, but in a conformal manner. That is, we have $$\lim_{b' \to b} \fra{d_y(b,b')}{d_x(b,b')} = e^{\beta_b(x,y)}$$ for a certain number $\beta_b(x,y)$ defined in the following way: if $r$ is a geodesic ray in $T$, then we denote by $\beta_r(x,y)$ the limit of $d(x,r(t)) - d(y,r(t))$ as $t$ tends to infinity, which exists and only depends on $r(+\infty)$. We call it the \textbf{horospheric distance} between $x$ and $y$ with respect to $r(+\infty)$. 

We can now define measures on $\partial X$. The Hausdorff dimension of $(\partial T, d_x)$ can easily be calculated, and is $\ln(d-1)$. Moreover, if we denote by $\mu_x$ the normalized $\ln(d-1)$-dimensional Hausdorff measure, isotropy around $x$ implies that $\mu_x(B) = \frac{1}{d(d-1)^{i-1}}$ if $B$ is a ball of radius $e^{-i}$ for the distance $d_x$. The map $$\begin{array}{rcl}\mu : T &\rightarrow &\cali{M}^1(\partial T)\\
x &\mapsto &\mu_x\\
\end{array}$$ is $\isom(T)$ equivariant in the sense that $\forall \gamma \in \isom(T)$, $(\partial \gamma)_* \mu_x = \mu_{\gamma(x)}$. Adding everything up, we get:

\begin{fact} The action $\isom(T) \curvearrowright (\partial T,\mu_{x_0})$ is a quasi-invariant action, and we have the formula \[\forall \gamma \in \isom(T),\ \  \forall b \in \partial T,\ \  \fra{\diff(\partial \gamma)_* \mu_{x_0}}{\diff\mu_{x_0}}(b) = (d-1)^{\beta_b(x_0,\gamma^{-1}(x_0))} =: c_T(\gamma,b).\]
\end{fact}

Let us recall that $G := \sln_2(\kk_0) \times \sln_2(\kk_\infty)$ acts componentwise on the product $T_0 \times T_\infty$ of the Bruhat-Tits trees, which are $(q+1)$-regular. Now, fix two vertices $v_0$ and $v_\infty$ in $T_0$ and $T_\infty$ and consider the product action $G \curvearrowright (\partial T_0 \times \partial T_\infty,\mu_{v_0} \otimes \mu_{v_\infty})$. 

From the above discussion, the following proposition is obvious.

\begin{prop}\label{prop: quasi-invariant and continuous cocycle}
In this setting, the product measure $\mu_{v_0}\otimes \mu_{v_\infty}$ is quasi-invariant under the action $G \curvearrowright \partial T_0 \times \partial T_\infty$, and the Radon-Nikodym cocycle is continuous.
\end{prop}

Let us denote $\pi$ the Koopman representation associated to this action, and $\Xi$ the Harish-Chandra function.

\subsection{Estimates on the growth of $\Gamma$} In this section, we provide the estimates on $\vert C^\Gamma_n \vert$ we need. To do so, it is useful to estimate the cardinal of sets of vertices inside balls in a product of two trees. We keep notation from Section \ref{defi}. 

We will now compute the number of elements of \[B_n := \left\{(x,y) \in V(T_0) \times V(T_\infty) \tq d_{\mathbb{I}}\left((v_0,v_\infty),(x,y)\right)\leq n \right\}.\] If $a \in \{0,\infty\}$, $i \in \nn$, $x \in T_a$, let us denote $S_a(x,i) := \{y \in T_a \tq d_a(x,y) = i\}$ and $B_a(x,i) := \cup_{j = 0,\dots,i} S_a(x,j)$.

\begin{lem}[Ball counting]\label{lemma: ball counting} There are $A,B,C \in \rr$ such that $A \not = 0$ and $$\forall n \in \nn,\ \ \vert B_n \vert = \left(An+B\right)(d-1)^n + C.$$
\end{lem}

\begin{proof} To make the calculation more readable, we set $D := \frac{d}{d-2}$. We first observe that $$B_n = \bigsqcup^n_{i = 0} \bigsqcup_{x \in S_0(v_0,i)} \{x\}\times B_\infty(v_\infty,n-i)$$ so that $\vert B_n \vert = \sum^n_{i=0} s_i b_{n-i}$ where we denote, for $i,j \in \nn$, $s_i := \vert S_0(v_0,i)\vert$ and $b_{j} := \vert B_\infty(v_\infty,j)\vert$.

We have that $\forall i,j \in \nn$, \[\begin{gathered}
\forall i \in \nn,\ \ s_i = \left\{\begin{array}{rl} d(d-1)^{i-1} &\mbox{if } i \geq 1\\1 &\mbox{if }i = 0\\
\end{array}\right. \\
\forall j \in \nn,\ \ b_j = \sum^j_{i = 0} s_i = 1 + D\left((d-1)^j - 1\right)\\ 
\forall i \geq 1,\ \ s_i b_{n-i} = D\left(d(d-1)^{n-1} - 2(d-1)^{i-1}\right)\end{gathered}\] and we get, $\forall n \in \nn$, \[\begin{array}{rcl}
\vert B_n \vert = \sum^n_{i = 0} s_i b_{n-i} &= &\displaystyle b_n + \sum^n_{i = 1} s_i b_{n-i}\\ 
&= &\displaystyle 1 + D((d-1)^n -1) + \sum^n_{i=1} D \left[d(d-1)^{n-1} - 2 (d-1)^{i-1}\right]\\
&= &\displaystyle 1 - D + D(d-1)^n + Ddn(d-1)^{n-1} - 2D \sum^n_{i= 1} (d-1)^{i-1}\\
&= &\displaystyle 1 - D + D(d-1)^n + Ddn(d-1)^{n-1} - \frac{2D}{d-2} (d-1)^n + \frac{2D}{d-2}\\
&= &\displaystyle (d-1)^n \left[n\frac{Dd}{d-1} + D - \frac{2D}{d-2}\right] + 1 - D + \frac{2D}{d-2}\\
\end{array}\]

so we choose \[\begin{gathered}
A := \frac{Dd}{d-1} = \frac{d^2}{(d-2)(d-1)}\\
B := D - \frac{2D}{d-2} = \frac{d(d-4)}{(d-2)^2}\\
C := 1 - D + \frac{2D}{d-2} = 1 - B
\end{gathered}
.\]
\end{proof}

\begin{prop}\label{groupgrowth} We have that \[\mu(G_n) \ll n(d-1)^n.\]
\end{prop}

\begin{proof} Consider the map \[\begin{array}{rcl}
\theta : G &\rightarrow &T_0 \times T_\infty\\
g &\mapsto &g(v_0,v_\infty).\\
\end{array}\]Then \[G_n = \theta^{-1}(B_n) = \sqcup_{y \in B_n} \theta^{-1}(\{y\}).\]Each of these fibers, if it is nonempty, is a $G_0$-left coset, so it has measure $\mu_G(G_0)$. So, according to Lemma \ref{lemma: ball counting}, \[\mu_G(G_n) \leq \vert B_n \vert \mu_G(G_0) \ll n(d-1)^n.\]
\end{proof}

\begin{prop} \label{growthestimates} \label{estimatesrecall}There is $c$ such that for all $n \in \nn^*$, \[\mu\left(C^G_n\right) \leq c \left\vert C^\Gamma_n \right\vert.\]
\end{prop}

\begin{proof} Let us recall that $\Gamma$ is an irreducible lattice in $G$, according to \cite[p. 1]{MARGU}. Since the action of $G$ on $G/\Gamma$ is mixing, the mean ergodic theorem holds, and therefore, we can apply \cite[Lemma 6.7, p. 79]{GORNEV}.
\end{proof}

\subsection{Estimates on the Harish-Chandra function}

In order to apply the criterion, we need to compute the Harish-Chandra function associated to the quasi-invariant action $\Gamma \curvearrowright \partial T_0 \times \partial T_\infty$. Since this is a product action, is is enough to calculate the Harish-Chandra functions on the factors:

\begin{lem} \label{productaction}The Harish-Chandra of a product of actions is the product of the Harish-Chandra functions on the factors. In particular, \[\forall (g_0,g_\infty) \in G,\ \ \Xi(g_0,g_\infty) = \int_{\partial T_0} c_{T_0}(g^{-1}_0,b)^{\frac{1}{2}}\diff \mu_{v_0}(b)\int_{\partial T_\infty} c_{T_\infty}(g^{-1}_\infty,b)^{\frac{1}{2}}\diff \mu_{v_\infty}(b).\]
\end{lem}

\begin{proof} Just apply Fubini's theorem.
\end{proof}

Our goal is now to compute $\int_{\partial T} c_T(\gamma^{-1},b)^{\frac{1}{2}} \diff \mu_{x_0}(b)$ for $T$ a $d$-regular tree, $\gamma \in \isom(T)$, $x_0$ a vertex of $T$ and $\mu_{x_0}$ the boundary measure on $\partial T$ associated to $x_0$. As we shall see, $b \mapsto c_T(\gamma,b)$ is piecewise constant, so we will suitably partition $\partial T$.

Let us define $S_{n,\gamma} := \{y \in T_d \tq d(x_0,y) = n,\ \  l([x_0,y] \cap [x_0,\gamma^{-1}(x_0)]) = l([x_0,y]) - 1\}$, and, for $y \in T$, $\cali{O}_y := \{\xi \in \partial T \tq y \in [x_0,\xi)\}$.

\begin{lem} With the above notation, let $n := d(x_0,\gamma^{-1}(x_0))$. The following properties hold true.
\begin{enumerate}
\item Assume $i < n$, $y \in S_{i,\gamma}$, $b \in \cali{O}_y$, $b' \in \partial T$. Then $$b' \in \cali{O}_y \Leftrightarrow (b\vert b')_{x_0} > i-1 \Leftrightarrow d_{x_0}(b,b') < e^{-i+1}.$$
\item $\partial T = \cali{O}_{\gamma^{-1}(x_0)} \sqcup \bigsqcup^n_{i=1} \left(\bigsqcup_{y \in S_{i,\gamma}} \cali{O}_y\right)$,
\item $\forall i \in \{2,...,n\}$, $\vert S_{i,\gamma(x_0)} \vert = d-2$, and $\vert S_{1,\gamma(x_0)} \vert = d-1$,
\item $\forall i \in \{1,...,n\}$, $\forall y \in S_{i,\gamma(x_0)}$, $\forall \xi \in \cali{O}_y$, $\beta_{\xi}(x_0,\gamma^{-1}(x_0)) = 2(i-1) - n$,
\item $\forall \xi \in \cali{O}_{\gamma^{-1}(x_0)}$, $\beta_{\xi}(x_0,\gamma^{-1}(x_0)) = n$.
\end{enumerate}
\end{lem}

\begin{proof}
\begin{enumerate}
\item We have that $b'\not\in\cali{O}(y)$ if and only if $l([x_0,b') \cap [x_0,b)) \leq l([x_0,y]) - 1 = i - 1$, which proves $(1)$.
\item The sets in the union are clearly disjoint, so it is enough to show that $\partial T$ is the mentioned union. Let $b \in \partial T$, and $r : \rr_+ \rightarrow T$ be the geodesic joining $x_0$ to $b$. Let $$t := \max\{t \in \rr_+ \tq r(t) \in [x_0,\gamma^{-1}(x_0)]\}.$$ Then $y := r(t+1) \in S_{t+1,\gamma}$ and $b \in \cali{O}_y$.
\item 4) and 5) are straightforward.
\end{enumerate}
\end{proof}

\begin{prop}\label{harishchandraarbre} Let $\gamma \in \isom(T_d)$ and denote $n := d(x_0,\gamma(x_0))$. Let $q=d-1$. We have $$\displaystyle\int_{\partial T} c(\gamma,b)^{\frac{1}{2}} \diff\mu(b) = \left(1 + \frac{q-1}{q+1}\mbox{ }n\right)q^{-\frac{n}{2}}.$$
\end{prop}

\begin{proof} Using the information collected in the above lemma, we do the following calculation:

\[\begin{array}{rcl}
\displaystyle\int_{\partial T} c(\gamma,b)^{\frac{1}{2}} \diff\mu(b) &= &\displaystyle\int_{\cali{O}_{\gamma^{-1}x_0}} c(\gamma,b)^{\frac{1}{2}}\diff\mu(b) + \sum^n_{i = 1} \sum_{y \in S_{i,\gamma}} \int_{\cali{O}_y} c(\gamma,b)^{\frac{1}{2}}\diff\mu(b)\\
&= &\displaystyle\mu\left(\cali{O}_{\gamma^{-1}x_0}\right)(d-1)^{\frac{n}{2}} + \sum^n_{i = 1} \sum_{y \in S_{i,\gamma}} \mu\left(\cali{O}_y\right)(d-1)^{i - 1 - \frac{n}{2}}\\
&= &\fra{1}{d(d-1)^{n-1}} (d-1)^{\frac{n}{2}} + \sum^n_{i = 1} \sum_{y \in S_{i,\gamma}} \fra{1}{d(d-1)^{i-1}}(d-1)^{i - 1 - \frac{n}{2}}\\
&= &\frac{1}{d}\displaystyle\left[(d-1)^{-\frac{n}{2} + 1} + \sum^n_{i = 1} \sum_{y \in S_{i,\gamma}} (d-1)^{-\frac{n}{2}}\right]\\
&= &\frac{1}{d}\displaystyle(d-1)^{-\frac{n}{2}}\left[d-1 + \sum^n_{i=1} \left\vert S_{i,\gamma}\right\vert\right]\\
&= &\frac{1}{d}\displaystyle(d-1)^{-\frac{n}{2}}\left[d-1 + d- 1 + (n-1)(d-2)\right]\\
&= &\displaystyle\left(1 + \frac{d-2}{d}n\right)(d-1)^{-\frac{n}{2}}\\
\end{array}\]
\end{proof}

\begin{prop}[The Harish-Chandra estimate] \label{lemma: HC}Let $\Xi$ be the Harish-Chandra function $\Xi$ associated to the action $G \curvearrowright (\partial T_0 \times \partial T_\infty, \mu_{x_0} \otimes \mu_{x_\infty})$. We then have
\begin{enumerate}
\item  for all $g := (g_0,g_\infty) \in G$,
\[\Xi(g) = \displaystyle\left(1 + \frac{q-1}{q+1}L(g) + \left(\frac{q-1}{q+1}\right)^2L_0(g_0)L_\infty(g_\infty)\right)q^{-\frac{L(g)}{2}} ;\]
\item $\Xi : G \rightarrow \rr$ is continuous ;
\item we have the following estimate: for all $g \in G$ (and therefore, for every $g \in \Gamma$),
\[
\Xi(g) \ll L(g)^2(d-1)^{-\frac{L(g)}{2}}.
\]
\end{enumerate}
\end{prop}

\begin{proof} The last two claims follow immediately from the first. We apply together Proposition \ref{harishchandraarbre} and Lemma \ref{productaction}: let $g:=(g_0,g_\infty) \in G$. We then have \[\begin{array}{rcl}
\Xi(g) &= &\left(1+\fra{q-1}{q+1}L_0(g_0)\right)q^{-\frac{L_0(g_0)}{2}}\left(1+\fra{q-1}{q+1}L_{\infty}(g_\infty)\right)q^{-\frac{L_{\infty}(g_\infty)}{2}}\\
&= &\displaystyle\left(1 + \frac{q-1}{q+1}(L_0(g_0) + L_\infty(g_\infty)) + \left(\frac{q-1}{q+1}\right)^2L_0(g_0)L_\infty(g_\infty)\right)q^{-\frac{L(g)}{2}}.\\
\end{array}\]
\end{proof}

\subsection{Comparing  the decay of the Harish-Chandra function with the volume growth in the group $\sln_2(\mathbb F_q[X,X^{-1}])$}
We are ready to prove condition $(1)$ from the criterion for the group $\Gamma=\sln_2(\mathbb F_q[X,X^{-1}])$ and its action on 
$\partial T_0 \times \partial T_\infty$. 

\begin{prop}\label{first ingredient shalom trick}
We have that
\[
	\sup_{\gamma\in C^\Gamma_n}\Xi(\gamma) \sqrt{|C^\Gamma_n|}\ll n^{5/2}.
\]
\end{prop}

\begin{proof}

On the one hand, we apply Proposition \ref{lemma: HC} and we obtain:
\[
\sup_{\gamma\in C_n}\Xi(\gamma)\leq \left(1 + \frac{d-2}{d}\mbox{ }n+\frac{1}{2}\left(\frac{d-2}{d}\right)^2n^2\right)(d-1)^{-\frac{n}{2}}.
\]
Therefore, we have \[\sup_{\gamma\in C_n}\Xi(\gamma) \ll n^2(d-1)^{-\frac{n}{2}}.\] 
On the other hand, since $C^\Gamma_n\subset\Gamma_n\subset G_n$, applying Lemma \ref{lemma: ball counting}, we obtain (recall that $q=d-1$)
\[
	\sqrt{|C^\Gamma_n|}\ll \sqrt{n(d-1)^n} = n^{\frac{1}{2}}(d-1)^{\frac{n}{2}}.\]
So, we have,
\[
	\sup_{\gamma\in C^\Gamma_n}\Xi(\gamma) \sqrt{|C^\Gamma_n|}\ll n^{5/2}.
\]
\end{proof}

\section{Proof of Corollary \ref{coro trois}}

\label{nord}

\subsection{RD and amenable subgroups of exponential growth}

Let us recall three easy lemmas, proved in \cite{GARNCRD}, which hold for any finitely generated group $\Lambda$:

\begin{lem} 
\begin{enumerate}
\item If a finitely generated group $\Lambda$ has RD with respect to some length function, then it has RD with respect to the word length associated with any finite symmetric generating set.
\item If a discrete group $\Lambda$ has RD with respect to some length function $L$, then each subgroup $H \leq \Lambda$ has RD with respect to the induced length function $L_{\vert_H}$.
\item If an amenable finitely-generated group has RRD with respect to some length function $L$, then it has polynomial growth with respect to $L$.
\end{enumerate}
\end{lem}

These three lemmas are combined in the following criterion, useful to prove that some discrete groups do not have RD:

\begin{prop} \label{obstructamena}Let $\Lambda$ be a discrete group and $H$ be an amenable finitely-generated subgroup of $\Lambda$. Then if $H$ has exponential growth with respect to some of its word-lengths, then $\Lambda$ does not have RD with respect to any of its proper length functions.
\end{prop}

\subsection{The lamplighter subgroup}

Let us consider the subgroup \[H := \left\{\left(\begin{array}{cc}X^n &P\\0 & X^{-n}\\ \end{array}\right) \tq n \in \zz ,\ \  P \in A\right\}\]of $\sln_2(A)$ and let $S$ denote the finite subset of $H$ \[\left\{\left(\begin{array}{cc}X &0\\0 & X^{-1}\\ \end{array}\right),\left(\begin{array}{cc}X^{-1} &0\\0 & X\\ \end{array}\right),\left(\begin{array}{cc}1 &\pm 1\\0 & 1\\ \end{array}\right),\left(\begin{array}{cc}1 &\pm X\\0 & 1\\ \end{array}\right)\right\}.\]

The following proposition is a routine exercise for people working in geometric group theory. We give the proof for readers with a different background.

\begin{prop} \label{lamplighter} The subgroup $H$ is amenable, $S$ is a symmetric generating set of $H$, and $H$ has exponential growth with respect to the word-length associated to $S$.
\end{prop}

\begin{proof} If $P \in \ff_q[X,X^{-1}]$, define $$\gamma(P) := \left(\begin{array}{cc}1 &P\\
0&1\\
\end{array}\right)$$ so that $\gamma : \ff_q[X, X^{-1}] \rightarrow H$ is a morphism. Define also \[
\psi\left(
\left(\begin{array}{cc}X^n &P\\0 & X^{-n}\\ \end{array}\right)\right):= n$$ so that $\psi : H \rightarrow \zz$ is a morphism. Then $$\xymatrix{
0 \ar[r] &\ff_q[X,X^{-1}] \ar^-{\gamma}[r] &H \ar^{\psi}[r] &\zz \ar[r] &0\\
}\] is a short exact sequence so $H$ is solvable, hence amenable.

Now let us prove that $H$ has exponential growth with respect to the word-length associated to $S$. Let $n \in \nn$, and $P := \sum^{n}_{i = 0} a_i X^{2i}$, where $a_i \in \{0,1\}$. There are $2^{n+1}$ such $P$, and we will prove that every $\gamma(P)$ can be written as a product of $3n + 1$ (or less) elements of $S$.

To do so, define $$A_0 := \left(\begin{array}{cc}
1 &a_n\\
0 &1\\
\end{array}\right)$$

\[A_{j+1} := \left(\begin{array}{cc}
X &0\\
0 &X^{-1}\\
\end{array}\right)A_{j}\left(\begin{array}{cc}
X^{-1} &0\\
0 &X\\
\end{array}\right)\left(\begin{array}{cc}
1 &a_{n-(j+1)}\\
0 &1\\
\end{array}\right)\]

It is straightforward to see that $A_n = \gamma(P)$, and by definition, $A_n$ is the product of (at most) $3n + 1$ elements of $S$.
\end{proof}

\begin{remasub} The subgroup $H$ is a said to be a lamplighter group.
\end{remasub}
We now prove Corollary 2.
\begin{proof} Proposition \ref{obstructamena} shows that if $(\Gamma,L)$ has the property RD, then it cannot contain an amenable finitely-generated exponential subgroup, and Proposition \ref{lamplighter} shows that $H$ is such a subgroup.
\end{proof}

\bibliographystyle{alpha}
\bibliography{bibliord}

\begin{thebibliography}{CPSC07}

\bibitem[Alb99]{ALBU}
P.~Albuquerque.
\newblock Patterson-{Sullivan} theory in higher rank symmetric spaces.
\newblock {\em GAFA, Geom. funct. anal.}, 9:1--28, 1999.

\bibitem[Ank87]{ANKER}
J.-P. Anker.
\newblock La forme exacte de l'estimation fondamentale de {Harish-Chandra}.
\newblock {\em C.R. Acad. Sci. Paris}, t.305:371--374, 1987.

\bibitem[BLP17]{BLP}
A.~Boyer, G.~Link, and C.~Pittet.
\newblock Ergodic properties of boundary representations.
\newblock {\em Ergod. Th. and Dynam. Sys.}, 39:2023--2047, 2017.

\bibitem[BM11]{BAMU}
U.~Bader and R.~Muchnik.
\newblock Boundary unitary representations - {Irreducibility} and rigidity.
\newblock {\em Journal of Modern Dynamics}, 5(1):49--69, 2011.

\bibitem[Bou95]{BOURD}
M.~Bourdon.
\newblock Structure conforme au bord et flot géodésique d'un
  {CAT}(-1)-espace.
\newblock {\em Enseign. Math}, 2(2):63--102, 1995.

\bibitem[Boy16]{BOYER}
A.~Boyer.
\newblock Equidistribution, ergodicity and irreducibility in {CAT}(-1) spaces.
\newblock {\em Geometry, Groups, Dynamics, to appear}, 2016.

\bibitem[BPL17]{BOYPIN}
A.~Boyer and A.~Pinochet~Lobos.
\newblock An ergodic theorem for the quasi-regular representation of the free
  group.
\newblock {\em Bulletin of the Belgian Mathematical Society}, 24:243--255,
  2017.

\bibitem[Cha17]{chatterjird}
I.~Chatterji.
\newblock {\em Around Langlands Correspondences}, volume 691 of {\em
  Contemporary Mathematics}, chapter Introduction to the rapid decay property,
  pages 55--72.
\newblock 2017.

\bibitem[CPSC07]{CHATPITSAL}
I.~Chatterji, C.~Pittet, and L.~Saloff-Coste.
\newblock Connected {Lie} groups and property {RD}.
\newblock {\em Duke Mathematical Journal}, 137:511--536, 2007.

\bibitem[Gar14]{GARNBOUND}
{Ł}. Garncarek.
\newblock Boundary representations of hyperbolic groups.
\newblock 2014.

\bibitem[Gar16]{GARNCRD}
{Ł}. Garncarek.
\newblock Mini-course: {Property of Rapid Decay}.
\newblock 2016.

\bibitem[GN10]{GORNEV}
A.~Gorodnik and A.~Nevo.
\newblock {\em The ergodic theory of lattice subgroups}.
\newblock Princeton University Press, 2010.

\bibitem[Haa79]{HAAGERUP}
U.~Haagerup.
\newblock An example of a nonnuclear {$C^*$}-algebra which has the metric
  approximation property.
\newblock {\em Invent. Math.}, 50:273--293, 1979.

\bibitem[Jol90]{JOLISS}
P.~Jolissaint.
\newblock Rapidly decreasing functions in reduced {$C^*$}-algebras of groups.
\newblock {\em Trans. Amer. Math. Soci.}, 317:167--196, 1990.

\bibitem[Laf00]{LAFFORGUE}
V.~Lafforgue.
\newblock A proof of property {(RD)} for discrete cocompact subgroups of
  $\sln_3(\rr)$ and $\sln_3(\cc)$.
\newblock {\em Journal of Lie Theory}, 10:255--267, 2000.

\bibitem[LMR00]{LUBMOZRAG}
A.~Lubotzky, S.~Mozes, and M.S. Raghunathan.
\newblock The word and {Riemannian} metrics on lattices of semisimple groups.
\newblock {\em Publications Mathématiques de l’Institut des Hautes
  Scientifiques}, 91:5--53, 2000.

\bibitem[Lob19]{APL}
A.~Pinochet Lobos.
\newblock {\em Théorèmes ergodiques, actions de groupes et représentations
  unitaires}.
\newblock PhD thesis, Université {Aix-Marseille}, 2019.

\bibitem[Mar91]{MARGU}
G.~A. Margulis.
\newblock {\em Discrete Subgroups of Semisimple Lie Groups}.
\newblock Springer-Verlag, 1991.

\bibitem[Per09]{Perrone}
M.~Perrone.
\newblock Radial rapid decay property for cocompact lattices.
\newblock {\em Journal of Functional Analysis}, 256:3471--3489, 2009.

\bibitem[Ser08]{SERRE}
J.~P. Serre.
\newblock {\em Trees}.
\newblock Cambridge University Press, 2008.

\bibitem[Sha00]{Shalom2000}
Y.~Shalom.
\newblock Rigidity, unitary representations of semisimple groups, and
  fundamental groups of manifolds with rank one transformation group.
\newblock {\em Annals of Mathematics. Second Series}, 152(1):113--182, 2000.

\bibitem[Val97]{VALETTE97}
A.~Valette.
\newblock On the {Haagerup} inequality and groups acting on
  {$\tilde{A}_n$}-buildings.
\newblock {\em Annales de l'{Institut Fourier}}, 47(4):1195--1208, 1997.

\bibitem[Val02]{VALETTEBAUMCONNES}
A.~Valette.
\newblock {\em An introduction to the {Baum-Connes conjecture}. From notes
  taken by Indira Chatterji, with an appendix by Guido Mislin.}
\newblock Lecture notes in Mathematics, ETH Zürich, Birkhäuser, 2002.

\end{thebibliography}

\end{document}